\documentclass[11pt]{amsart}
\usepackage[margin=1.35in]{geometry}
\usepackage{amscd,amsmath,amsxtra,amsthm,amssymb,stmaryrd,xr,mathrsfs,mathtools,enumerate,commath, comment}
\usepackage{stmaryrd}
\usepackage{xcolor}
\usepackage{commath}
\usepackage{comment}
\usepackage{tikz-cd}
\usepackage{longtable} 
\usepackage{pdflscape} 
\usepackage{booktabs}
\usepackage{hyperref}
\definecolor{vegasgold}{rgb}{0.77, 0.7, 0.35}
\definecolor{darkgoldenrod}{rgb}{0.72, 0.53, 0.04}
\definecolor{gold(metallic)}{rgb}{0.83, 0.69, 0.22}
\hypersetup{
 colorlinks=true,
 linkcolor=darkgoldenrod,
 filecolor=brown,      
 urlcolor=gold(metallic),
 citecolor=darkgoldenrod,
 }

\usepackage[all,cmtip]{xy}

\DeclareFontFamily{U}{wncy}{}
\DeclareFontShape{U}{wncy}{m}{n}{<->wncyr10}{}
\DeclareSymbolFont{mcy}{U}{wncy}{m}{n}
\DeclareMathSymbol{\Sh}{\mathord}{mcy}{"58}
\usepackage[T2A,T1]{fontenc}
\usepackage[OT2,T1]{fontenc}

\newtheorem{theorem}{Theorem}[section]
\newtheorem{lemma}[theorem]{Lemma}

\newtheorem*{theorem*}{Theorem}
\newtheorem*{ass*}{Assumption}

\newtheorem{corollary}[theorem]{Corollary}
\usepackage{enumitem}

\newtheorem{question}[theorem]{Question}

\newcommand{\Z}{\mathbb{Z}}

\newcommand{\Q}{\mathbb{Q}}
\newcommand{\F}{\mathbb{F}}

\newcommand{\op}[1]{\operatorname{#1}}

\numberwithin{equation}{section}

\begin{document}

\title[sums of cubes in the $\Z_p$-extension]{Integers that are sums of two cubes in the cyclotomic $\Z_p$-extension}

\author[A.~Ray]{Anwesh Ray}
\thanks{Author's Orcid ID: 0000-0001-6946-1559}
\address[Ray]{Chennai Mathematical Institute, H1, SIPCOT IT Park, Kelambakkam, Siruseri, Tamil Nadu 603103, India}
\email{anwesh@cmi.ac.in}

\keywords{ranks of elliptic curves, diophantine applications of Iwasawa theory,  representing a number as a sum of two cubes}
\subjclass[2020]{Primary 11R23, Secondary 11G05, 11D25}

\maketitle

\begin{abstract}
 Let $n$ be a cubefree natural number and $p\geq 5$ be a prime number. Assume that $n$ is not expressible as a sum of the form $x^3+y^3$, where $x,y\in \mathbb{Q}$.  In this note,  we study the solutions (or lack thereof) to the equation $n=x^3+y^3$,  where $x$ and $y$ belong to the cyclotomic $\mathbb{Z}_p$-extension of $\mathbb{Q}$. As an application, consider the case when $n$ is not a sum of rational cubes. Then, we prove that $n$ cannot be a sum of two cubes in certain large families of prime cyclic extensions of $\mathbb{Q}$.
\end{abstract}

\section{Introduction}

\par Given a natural number $n$,  and $k\in \Z_{\geq 2}$, it is of natural interest to investigate if $n$ may be represented as a sum of two $k$-th powers
\[n=x^k+y^k,\] where $x$ and $y$ are rational numbers. The problem has been completely solved by Euler when $k=2$,  and in this case,  it is easy to see that $n$ is a sum of two rational squares if and only if it is a sum of two integers squares. It turns out that $n$ is a sum of two integer squares if all prime factors $\ell$ of $n$ such that $\ell\equiv 1\mod{4}$ have even exponent in the factorization of $n$. It follows from this description that the natural density of numbers that can be represented as a sum of two rational squares is $0$. On the other hand,  it is possible for $n$ to be a sum of two rational cubes without being a sum of two integer cubes. The smallest example is 
\[6=\left(\frac{17}{21}\right)^3+\left(\frac{37}{21}\right)^3. \] It is expected that the density of natural numbers $n$ that can be represented as a sum of two rational cubes is $\frac{1}{2}$, cf. \cite{alpoge2022integers}.  There has been much work done on the problem of determining when $n$ is a sum of rational cubes and the problem goes back to work of Sylvester \cite{Sylvester}, Selmer \cite{Selmer}, Satg\'e \cite{satge} and Lieman \cite{Lieman}.  More recently, Alp\"oge, Bhargava and Schnidman \cite{alpoge2022integers} showed that the lower density of natural numbers $n$ that are (resp. are not) a sum of two rational cubes is $\geq \frac{2}{21}$ (resp $\geq \frac{1}{6}$).  

\par In this article we consider a related question. Suppose $n$ is \emph{not} expressible as a sum of two rational cubes and $L$ is an algebraic extension of $\Q$. We seek conditions for $n$ to not equal $x^3+y^3$, for all $x,y\in L$. The extensions $L$ we shall consider are infinite cyclotomic extensions that have come to prominence thanks to celebrated work of Iwasawa, cf. \cite{iwasawa} and \cite[Ch. 13]{washington1997introduction}. Fix an algebraic closure $\bar{\Q}$ of $\Q$. Let $p\geq 5$ be a prime number and $\Q(\mu_{p^\infty})\subset \bar{\Q}$ be the infinite cyclotomic extension of $\Q$ generated by all $p$-power roots of unity.  Via the cyclotomic character,  the Galois group $\op{Gal}(\Q(\mu_{p^\infty})/\Q)$ is isomorphic to $\Z_p^\times$,  the units in $\Z_p$. This group decomposes into a product $\Delta\times U$,  where $\Delta\simeq (\Z/p\Z)^\times$ and $U=1+p\Z_p$ is the subgroup of principal units in $\Z_p^\times$.  The $p$-adic logarithm gives an isomorphism of $U$ with the additive group $\Z_p$.  Let $\Q_\infty^{(p)}$ be the subfield of $\Q(\mu_{p^\infty})$ that is fixed by $\Delta$.  By construction,  $\Q_\infty^{(p)}$ is a Galois extension of $\Q$ with Galois group isomorphic to $\Z_p$.  This extension is referred to as the cyclotomic $\Z_p$-extension of $\Q$. 

\par Being a pro-$p$ abelian extension of $\Q$,  the cyclotomic $\Z_p$-extension satisfies some nice properties.  All primes are finitely decomposed in $\Q_\infty^{(p)}$ and the only prime that ramifies is $p$.  It is the unique $\Z_p$-extension of $\Q$.  Let $E_n$ denote the elliptic curve defined by $x^3+y^3=n z^3$.  This elliptic curve is seen to have no nontrivial rational torsion points. Therefore $n$ is a sum of two rational cubes if and only if the rank of $E_n(\Q)$ is positive.  The growth of ranks of elliptic curves in $\Z_p$-extensions of number fields was studied by Mazur, cf. \cite{mazur1972rational}.  Kato \cite{kato2004p} and Rohrlich \cite{rohrlich} proved that given an elliptic curve $E_{/\Q}$,  the rank of $E(\Q_n^{(p)})$ is bounded as $n$ goes to $\infty$.  Moreover the upper bound can be related to the algebraic structure of Selmer groups defined over $\Q_\infty^{(p)}$.  The conjectural properties of these Selmer groups were described by Mazur,  and later Rubin \cite{rubinbsd} and Kato \cite{kato2004p} showed that they are cotorsion as modules over the Iwasawa algebra. The structure of such groups are closely related to analytic objects known as $p$-adic L-functions.  

\par In this article,  we employ techniques from the Iwasawa theory of Selmer groups to study the following question.
\begin{question}
Let $n$ be a cubefree natural number which is not the sum of two rational cubes.  Let $p\geq 5$ be a prime at which the elliptic curve 
\[E_n: x^3+y^3=n z^3\] has good ordinary reduction. Then,  can $n$ be a sum of the form $x^3+y^3$, where $x,y\in \Q_\infty^{(p)}$. 
\end{question}

\begin{theorem}\label{main thm}
    Let $n$ be a cubefree natural number which is not the sum of two cubes in $\Q$. Assume that the Tate-Shafarevich group $\Sh(E_n/\Q)$ is finite.  Let $S$ be the set of prime numbers $p\geq 5$ such that 
\begin{enumerate}[label=(\alph*)]
\item $E_n$ has good ordinary reduction at $p$,
\item $\Sh(E_n/\Q)[p]=0$,
\item $p$ does not divide $|\widetilde{E}_n(\F_p)|$, where $\widetilde{E}_n$ is the reduction of $E_n$ modulo $p$. 
\end{enumerate}Then the following assertions hold
\begin{enumerate}
    \item $S$ has density $ \frac{1}{2}$.
    \item For any prime $p\in S$,  $n$ is not a sum of two cubes in $\Q_\infty^{(p)}$.
\end{enumerate}
\end{theorem}

The assumption that $p\geq 5$ in Theorem \ref{main thm} is necessary since it is assumed that $E_n$ has good reduction at $p$. On the other hand, the curves $E_n$ have bad reduction at $2,3$. Our methods rely on the aforementioned results of Rubin and Kato, which is only valid for elliptic curves with good ordinary reduction at $p$.  Since the elliptic curves in question have complex multiplication,  the density of primes of good ordinary (resp. supersingular) reduction is $\frac{1}{2}$ (resp. $\frac{1}{2}$).  In effect, we show that the set of primes for which our result holds is of density $1$ in the set of primes of good ordinary reduction.  On the other hand, if $p$ is a prime of good ordinary reduction such that $\op{Sel}_{p^\infty}(E_n/\Q_\infty^{(p)})\neq 0$,  then this does not necessarily mean that $E_n(\Q_\infty^{(p)})\neq 0$,  since the contribution could as well arise from growth in the $p$-primary part of the Tate-Shafarevich group. Using Theorem \ref{main thm} we obtain the following unexpected result on the non-solvability of the equation $n=x^3+y^3$ in certain prime cyclic extensions of $\Q$. 

\begin{theorem}\label{aux thm}
    Let $n$ be a cube-free natural number which is not a sum of cubes in $\Q$ and let $p>7$ be a prime number in the set $S$ described in Theorem \ref{main thm}. Let $L/\Q$ be a Galois extension with Galois group $G:=\op{Gal}(L/\Q)$ and let $\Sigma$ be the primes $\ell$ that ramify in $L$. Assume that 
    \begin{enumerate}
        \item $G\simeq \Z/p\Z$, 
        \item the primes $\ell\in \Sigma$ do not divide $6pn$,
        \item for primes $\ell\in \Sigma$, we assume that $\widetilde{E}_n(\F_\ell)[p]=0$.
    \end{enumerate}
    Then, $n$ \emph{cannot} be represented as $x^3+y^3$ for $x,y\in L$.
\end{theorem}

In fact, our method shows that $n$ cannot be a sum of two cubes in the cyclotomic $\Z_p$-extension of $L$. This result is proven by showing that the Iwasawa $\mu$- and $\lambda$-invariants are trivial for the Selmer group considered over this $\Z_p$-extension. The key ingredient is an analogue of \emph{Kida's formula}, due to Hachimori and Matsuno \cite{HachimoriMatsuno}, which studies the growth of such Iwasawa invariants in $p$-extensions. Given an elliptic curve $E_{/\Q}$ with $\op{End}_{\bar{\Q}}(E)=\op{End}_{\Q}(E)$, it follows from the results of Mazur and Rubin \cite{Mazurrubindiophantine} on \emph{diophantine stability} that for a positive density of primes, there will exist infinitely many $L/\Q$ for which $E(L)=E(\Q)$. For elliptic curves $E_n$, our result applies to all prime numbers $p>7$ that lie in the explicitly defined set $S$.

\subsection{Outlook and related work} In recent years, there has been growing momentum in the study of solutions to diophantine equations in cyclotomic towers of number fields, bridging seemingly disparate themes in number theory. For instance, Freitas, Kraus and Siksek \cite{FKS1, FKS2} studied the Fermat equation $x^\ell+y^\ell=z^\ell$ in cyclotomic towers over $\Q$. The author and Weston merged diophantine methods with Galois cohomological ideas from Iwasawa theory \cite{RayWeston} to prove special cases of Hilbert's tenth problem in anticyclotomic towers of imaginary quadratic fields.  

\subsection{Organization} Including the introduction the article consists of three sections.  In the second section,  we introduce notation and preliminaries on the Iwasawa theory of elliptic curves.  In the final section,  we prove the Theorems \ref{main thm} and \ref{aux thm}. We finish with an explicit example to illustrate these results. 
\subsection*{Acknowledgement} The author wishes to thank the referee for the helpful report.

\section{Iwasawa theory of elliptic curves}
\par Let $n$ be a cube-free natural number and let $E_n$ be the elliptic curve defined by the Weierstrass equation 
\[E_n: x^3+y^3=nz^3.\]
It is well known that $E_n(\Q)_{\op{tors}}=0$ and thus, $n$ can be represented by a sum of two rational cubes if and only if $\op{rank}E_n(\Q)>0$. The Weierstrass form of the curve $E_n$ is given by $y^2=x^3-432n^2$. It is easy to see that $E_n$ has good reduction at all primes $p\nmid 6n$. For the rest of this section, $E$ is an elliptic curve defined over $\Q$. A finite Galois extension $L/\Q$ is said to be a $p$-extension if $\#\op{Gal}(L/\Q)$ is a power of $p$.  A Galois extension $L/\Q$ which is an increasing union of $p$-extensions is referred to as a pro-$p$ extension of $\Q$.  For a natural number $m$,  let $E[m]$ denote the $m$-torsion subgroup of $E(\bar{\Q})$ and 
\[\rho_{E,m}:\op{Gal}(\bar{\Q}/\Q)\rightarrow \op{GL}_2(\Z/m\Z)\]the associated Galois representation.  The field $\Q(E[m])$ is the fixed field of the kernel of $\rho_{E,m}$. We note that $\Q(E[m])$ is a Galois extension of $\Q$ whose Galois group can be identified with the image of $\rho_{E,m}$.  Given a number field $F$,  we set $F(E[m])$ to denote the composite of $F$ with $\Q(E[m])$.

\begin{lemma}\label{lemma 2.1}
    Let $n$ be a cube-free natural number and $L$ be a pro-$p$ extension of $\Q$. 
    Assume that $p\nmid 6n$ and that $p$ is totally ramified in $L$. Then, $n$ is a sum of two cubes $x, y\in L$ if and only if $\op{rank} E_n(L)>0$.
\end{lemma}
\begin{proof}
    Let $E:=E_n$, we show that $E(L)_{\op{tors}}=0$. Since $L$ is a union of finite $p$-extensions, we may assume without loss of generality that $L/\Q$ is a finite $p$-extension. It suffices to show that for all prime numbers $q$, we have that $E(L)[q]=0$. First consider the case when $q=p$. Since $L/\Q$ is a $p$-extension, it follows from \cite[Proposition 1.6.12]{NSW} that
    \[E(\Q)[p]=0\Rightarrow E(L)[p]=0.\]
    Next, we consider the case when $q\neq p$, then, the prime $p$ is unramified in $\Q(E[q])$. On the other hand, since $p$ is totally ramified in $L$, it follows that $L\cap \Q(E[q])=\Q$. As a consequence, we find that $\op{Gal}(L(E[q])/L)\simeq \op{Gal}(\Q(E[q])/\Q)$, and therefore, 
    \[E(\Q)[q]=0\Rightarrow E(L)[q]=0.\]
\end{proof}

Let $p$ be an odd prime number and $\Z_p$ denote the ring of $p$-adic integers. The \emph{cyclotomic $\Z_p$-extension} $\Q_\infty^{(p)}/\Q$ is the unique $\Z_p$-extension contained in $\Q(\mu_{p^\infty})$. Let $\Q_n^{(p)}$ be the extension of $\Q$ that is contained in $\Q_\infty^{(p)}$ such that $[\Q_n^{(p)}:\Q]=p^n$. We refer to $\Q_n^{(p)}$ as the \emph{$n$-th layer}. Let $K$ be a field of characteristic $0$ and $E$ be an elliptic curve over $\Q$. We set $\bar{K}$ to be a choice of algebraic closure of $K$ and $\op{G}_K$ denote the Galois group $\op{Gal}(\bar{K}/K)$. When it is clear from the context, set $E:=E(\bar{K})$ and $E[p^n]:=E(\bar{K})[p^n]$. Consider the exact sequence of $\op{G}_K$-modules
\[0\rightarrow E[p^n]\rightarrow E\xrightarrow{\times p^n} E\rightarrow 0, \]
associated to which is the exact sequence
\[0\rightarrow E(K)\otimes \left(\Z/p^n\Z\right)\rightarrow H^1(\op{G}_K, E[p^n])\rightarrow H^1(\op{G}_K, E)[p^n]\rightarrow 0.\] Taking direct limits as $n\rightarrow \infty$, obtain a short exact sequence 
\[0\rightarrow E(K)\otimes \left(\Q_p/\Z_p\right)\rightarrow H^1(\op{G}_K, E[p^\infty])\rightarrow H^1(\op{G}_K, E)[p^\infty]\rightarrow 0. \]
Let $F$ be a number field. For each non-archimedian prime $v$ of $F$, let $F_v$ denote the completion of $F$ at $v$. Choose an embedding $\iota_v:\bar{F}\hookrightarrow \bar{F}_v$ for each prime $v$, and \[\iota_v^* : \op{G}_{F_v}\hookrightarrow \op{G}_F\] the associated embedding. This embedding induces the restriction map 
\[\op{res}_v: H^1(\op{G}_F, E[p^\infty])\rightarrow H^1(\op{G}_{F_v}, E[p^\infty]).\]  The $p$-primary Selmer group over $F$ is defined as follows
\[\op{Sel}_{p^\infty}(E/F):=\left\{c\in H^1(\op{G}_\Q, E[p^\infty])\mid \op{res}_v(c)\in E(F_v)\otimes \Q_p/\Z_p\text{ for all primes }v \right\}.\] 
For a number field $F$, denote by $\Sh(E/F)$ the \emph{Tate-Shafarevich group}
\[\Sh(E/F):=\left\{H^1(\bar{F}/F, E)\rightarrow \prod_{v} H^1(\bar{F}_v/F_v, E)\right\}.\]
The Selmer group is an extension of the Tate-Shafarevich group, there is a short exact sequence 
\begin{equation}\label{selmer ses}0\rightarrow E(F)\otimes \Q_p/\Z_p\rightarrow \op{Sel}_{p^{\infty}}(E/F)\rightarrow \Sh(E/F)[p^{\infty}]\rightarrow 0.\end{equation}When $F/\Q$ is Galois, the Selmer group is a module over $\Z_p[\op{Gal}(F/\Q)]$. The Selmer group over $\Q_\infty^{(p)}$ is defined to be the following direct limit
\[\op{Sel}_{p^\infty}(E/\Q_\infty^{(p)}):=\lim_{n\rightarrow \infty}\op{Sel}_{p^\infty}(E/\Q_n^{(p)}). \]
Let $\Gamma$ denote the Galois group $\op{Gal}(\Q_\infty^{(p)}/\Q)$ and let $\gamma$ be a topological generator of $\Gamma$. The Iwasawa algebra $\Lambda$ is the completed group algebra
\[\Lambda:=\varprojlim_n \Z_p[\Gamma/\Gamma^{p^n}]. \] Set $T:=\gamma-1$ and identify $\Lambda$ with the formal power series ring $\Z_p\llbracket T \rrbracket $. The dual Selmer group \[\op{Sel}_{p^\infty}(E/\Q_\infty^{(p)})^\vee:=\op{Hom}_{\op{cnts}}\left(\op{Sel}_{p^\infty}(E/\Q_\infty^{(p)}), \Q_p/\Z_p\right)\] is the Pontryagin dual of $\op{Sel}_{p^\infty}(E/\Q_\infty^{(p)})$. It follows from results of Kato \cite{kato2004p} that $\op{Sel}_{p^\infty}(E/\Q_\infty^{(p)})^\vee$ is finitely generated and torsion as a module over $\Lambda$. 

\par Set $M$ to denote the dual Selmer group of $E$ over $\Q_\infty^{(p)}$. A polynomial $f(T)\in \Z_p[T]$ is \emph{distinguished} if it is a monic polynomial whose nonleading coefficients are divisible by $p$. By the structure theory of finitely generated and torsion modules over $\Lambda$, there is a homomorphism 
\[M\longrightarrow \left(\bigoplus_{i=1}^s \frac{\Lambda}{(p^{\mu_i})}\right)\oplus \left( \bigoplus_{i=1}^t \frac{\Lambda}{(f_j(T))} \right),\] whose kernel and cokernel are both finite (see \cite[Chapter 13]{washington1997introduction}). In the above decomposition, $f_j(T)$ is a distinguished polynomial. The \emph{characteristic element} is defined as follows
\[f_M(T):=\prod_i p^{\mu_i} \times \prod_j f_j(T).\] This element decomposes into a product 
\[f_M(T)=p^\mu\times g(T),\] where $\mu\in \Z_{\geq 0}$ and $g(T)$ is a distinguished polynomial. The $\mu$-invariant is the quantity $\mu=\sum_i \mu_i$ in the above decomposition. The $\lambda$-invariant is the degree of $g(T)$. The $\mu$ and $\lambda$-invariants of $\op{Sel}_{p^\infty}(E/\Q_\infty^{(p)})^\vee$ over $\Q_\infty^{(p)}$ are denoted $\mu_p(E)$ and $\lambda_p(E)$ respectively. Denote by $j_E$ the $j$-invariant of $E$.  
\begin{theorem}\label{thm 2.2}
    Let $p\geq 5$ be a prime number at which $E$ has good ordinary reduction such that 
    \begin{enumerate}
        \item $p$ does not divide $|\widetilde{E}(\F_p)|$, where $\widetilde{E}$ is the reduction of $E$ modulo $p$, 
        \item if $E$ has split multiplicative reduction at $\ell\neq p$, then, $p\nmid \op{ord}_\ell(j_E)$.
    \end{enumerate}
    Then there is a natural map
    \[\op{Sel}_{p^\infty}(E/\Q)\rightarrow \op{Sel}_{p^\infty} (E/\Q_\infty^{(p)})^\Gamma \] which is surjective. In particular, if $\op{Sel}_{p^\infty}(E/\Q)=0$, then, $\op{Sel}_{p^\infty}(E/\Q_\infty^{(p)})=0$.
\end{theorem}
\begin{proof}
    The above result is due to Greenberg, cf \cite[Proposition 3.8]{greenbergIwasawa97}.
\end{proof}

\par Let $\Pi$ be the set of all primes $p\geq 5$ at which $E$ has good ordinary reduction. For a subset $S\subset \Pi$ and let $S(x):=\{p\in S\mid p\leq x\}$. The natural density of $S$ is the limit 
\[\mathfrak{d}(S):=\lim_{x\rightarrow \infty} \frac{\# S(x)}{\# \Pi(x)}.\]

\begin{corollary}
    Let $E$ be an elliptic curve defined over $\Q$ such that $E(\Q)$ and $\Sh(E/\Q)$ are finite. Let $S_E$ be the set of primes $p\geq 5$ at which $E$ has good ordinary reduction and $\op{Sel}_{p^\infty}(E/\Q_\infty^{(p)})=0$. Then, we find that $\mathfrak{d}(S_E)=1$. 
\end{corollary}
\begin{proof}
    It is known that $p\nmid |\widetilde{E}(\F_p)|$ for a density $1$ set of primes $p$, cf. \cite{vkmurty}. The result then follows from Theorem \ref{thm 2.2}.
\end{proof}

\section{Proof of the main results}

\par In this section, we prove our main result. Throughout this section, $n$ is a cube-free natural natural number and let $E_n$ be the elliptic curve defined by the Weierstrass equation 
\[E_n: x^3+y^3=nz^3.\]

\begin{lemma}\label{lemma3.1}
    Let $p\nmid 6n$ be an odd prime number at which $E_n$ has good ordinary reduction and $\Q_\infty^{(p)}$ be the cyclotomic $\Z_p$-extension of $\Q$. The following conditions are equivalent
    \begin{enumerate}
        \item $n$ is not a sum of two cubes in $\Q_\infty^{(p)}$. 
        \item $\op{rank} E_n(\Q_\infty^{(p)})=0$. 
    \end{enumerate}
\end{lemma}
\begin{proof}
    The result follows from Lemma \ref{lemma 2.1}.
\end{proof}

\begin{proof}[Proof of Theorem \ref{main thm}]
    The elliptic curve $E:=E_n$ is a cubic twist of $E_1$ and thus has complex multiplication,  and in this setting, the set of primes $p$ at which $E$ has good ordinary reduction has density $\frac{1}{2}$.  By assumption,  $n$ is not a sum of two rational cubes,  hence,  $E(\Q)=0$.  After a change of coordinates, $E_n$ is given by the equation $y^2=x^3-432 n^2$. Thus in particular, there are no primes at which $E$ has multiplicative reduction. Let $p\in S$,  then,  since $E(\Q)=0$ and $\Sh(E/\Q)[p]=0$, it follows from \eqref{selmer ses} that $\op{Sel}_{p^\infty}(E/\Q)=0$.  Theorem \ref{thm 2.2} then implies that $\op{Sel}_{p^\infty}(E/\Q_\infty^{(p)})=0$ for all primes $p\in S$. The short exact sequence 
    \[0\rightarrow E(\Q_\infty^{(p)})\otimes \Q_p/\Z_p\rightarrow \op{Sel}_{p^\infty}(E/\Q_\infty^{(p)})\rightarrow \Sh(E/\Q_\infty^{(p)})[p^\infty]\rightarrow 0\] implies that $E(\Q_\infty^{(p)})\otimes \Q_p/\Z_p=0$. This gives that the rank of $E(\Q_\infty^{(p)})$ is $0$.  It then follows from Lemma \ref{lemma3.1} that for all $p\in S$,  $E(\Q_\infty^{(p)})=0$,  i.e., $n$ cannot be represented as $x^3+y^3$ for $x,y\in \Q_\infty^{(p)}$. The density of primes $p$ at which $\widetilde{E}(\F_p)[p]\neq 0$ is $0$ (cf. \cite{vkmurty}). This is indeed weaker than what is predicted by the Lang-Trotter conjecture, a precise unconditional asymptotic is obtained in \emph{loc. cit.} It follows that the density of $S$ is equal to the density of primes $p$ at which $E$ has good ordinary reduction.  Therefore,  the density of $S$ is equal to $\frac{1}{2}$. 
\end{proof}

\begin{proof}[Proof of Theorem \ref{aux thm}]
    Setting $E:=E_n$, it suffices to show that $\op{rank} E(L)=0$ and $E(L)_{\op{tors}}=0$. Let $L_\infty^{(p)}:=L\cdot \Q_\infty^{(p)}$, i.e., the cyclotomic $\Z_p$-extension of $L$. We recall from the proof of Theorem \ref{main thm} that $\op{Sel}(E/\Q_\infty^{(p)})=0$. In particular, the $\mu$- and $\lambda$-invariants of $\op{Sel}(E/\Q_\infty^{(p)})$ are equal to $0$. It follows from our assumptions on $p$ and $n$ and the main result of \cite{HachimoriMatsuno} that the $\mu$- and $\lambda$-invariants of $\op{Sel}(E/L_\infty^{(p)})$ are both equal to $0$ (see \cite[Lemma 4.10]{RayASDS} and its proof for further details). This implies in particular that the rank of $E(L)$ is $0$ (cf. \cite[Theorem 1.9]{GreenbergIwasawatheory}).

    \par Since $n$ is not expressible as a sum of rational cubes, we have that $E(\Q)=0$, and in particular, $E(\Q)_{\op{tors}}=0$. It follows from \cite[Theorem 7.2]{GJandNajman} that since $p>7$, $E(L)_{\op{tors}}=E(\Q)_{\op{tors}}=0$. Therefore, $E(L)=0$ and hence, $n$ cannot be represented as $x^3+y^3$ for any $x,y\in L$.
\end{proof}

\subsection*{An example} Let us consider an explicit example. The number $n=3$ is not a sum of two rational cubes. This can be seen by considering the elliptic curve $E:=E_3: x^3+y^3=3$. In standard form, the equation of this curve is $E:y^2=x^3-3888$. The data provided to us on LMFDB (see \href{https://www.lmfdb.org/EllipticCurve/Q/243/b/1}{243b2}) confirms that this elliptic curve has trivial Mordell--Weil group over $\Q$. This elliptic curve has good ordinary reduction at $7$, and $\# \widetilde{E}(\F_7)=3$. In particular, $7$ does not divide $\# \widetilde{E}(\F_7)=3$. Moreover, the exact order of the Tate--Shafarevich group is $1$. Thus, we find that $7$ is contained in the set $S$ defined in Theorem \ref{main thm}. The Theorem thus asserts that $3$ is not a sum of cubes in the cyclotomic $\Z_7$-extension of $\Q$. Now let $L/\Q$ be the unique $\Z/7\Z$-extension of $\Q$ which is contained in $\Q(\mu_{29})$. The set of primes $\Sigma$ introduced in Theorem \ref{aux thm} is the single prime $\{29\}$. We find that 
\[\# \widetilde{E}(\F_{29})=29+1-a_{29}(E)=30.\] Therefore, we have that $\widetilde{E}(\F_{29})[7]=0$. Thus, the conditions of Theorem \ref{aux thm} are satisfied, and we conclude that $3$ is not representable as $x^3+y^3$ for $x,y\in L$.

\bibliographystyle{alpha}
\bibliography{references}
\end{document}